\documentclass[a4paper,11pt]{article} 

\usepackage{hyperref}
\usepackage[affil-sl]{authblk}

\usepackage{titlesec}
\titleformat{\section}
  {\Large\center\bfseries\scshape}
  {\thesection.}{.7em}{}
\titlespacing*{\section}{0pt}{3.5ex plus 0ex minus 0ex}{1.5ex plus 0ex}

\titleformat{\subsection}
  {\center\bfseries\scshape}
  {\thesubsection.}{.7em}{}
\titlespacing*{\subsection}{0pt}{3.5ex plus 0ex minus 0ex}{1.5ex plus 0ex}

\usepackage[OT2, T1]{fontenc}
\usepackage[russian,english]{babel}
\usepackage{amsfonts}
\usepackage{mathrsfs}
\usepackage{bbm}
\usepackage{latexsym}
\usepackage{amssymb}
\usepackage{amsmath}

\newcommand\blfootnote[1]{%
  \begingroup
  \renewcommand\thefootnote{}\footnote{#1}%
  \addtocounter{footnote}{-1}%
  \endgroup
}
\usepackage{epigraph}
\usepackage{enumitem}
\setlist{nolistsep}
\usepackage{amsthm}
\usepackage[capitalize]{cleveref}

\newtheoremstyle{plain}
{3mm}	
{3mm}	
{\slshape}	
{}	
{\color{Blue}\bfseries}	
{.}	
{.5em}	
{}	
\newtheoremstyle{definition}
{2mm}
{2mm}
{}
{}
{\color{Blue}\bfseries}
{.}
{.5em}
{}
\theoremstyle{plain}

\newtheorem{Theorem}{Theorem}
\newtheorem{Question}{Question}

\newtheorem{Lemma}[Theorem]{Lemma}
\newtheorem{Proposition}[Theorem]{Proposition}
\newtheorem{Corollary}[Theorem]{Corollary}
\theoremstyle{definition}

\theoremstyle{plain} 
\newcounter{MainTheoremCounter}

\theoremstyle{plain}
\newtheorem*{namedthm}{\namedthmname}
\newcounter{namedthm}
	

\numberwithin{equation}{section}

\usepackage{xcolor}

\definecolor{Scarlet}{rgb}{0.78, 0.11, 0.0}
\definecolor{Blue}{rgb}{0.0, 0.0, 0.0}
\definecolor{Green}{rgb}{0.39, 0.71 ,0.0}

\hypersetup{citecolor = Scarlet,colorlinks,
			linkcolor = black,
			urlcolor = Scarlet}

\newcommand{\N}{\mathbb{N}}
\newcommand{\Z}{\mathbb{Z}}
\newcommand{\R}{\mathbb{R}}
\newcommand{\Q}{\mathbb{Q}}

\newcommand{\define}[1]{\textcolor{Blue}{\emph{#1}}}

\renewcommand{\epsilon}{\varepsilon}
\renewcommand{\leq}{\leqslant}
\renewcommand{\geq}{\geqslant}

\renewcommand{\setminus}{\backslash}

\newcommand{\1}{1}

\newcommand{\Hardy}{\mathcal{H}}
\newcommand{\LE}{\mathcal{L}}

\begin{document}

\title{On the density of coprime tuples of the form $(n,\lfloor f_1(n)\rfloor,\ldots,\lfloor f_k(n)\rfloor)$,  where $f_1,\ldots,f_k$ are functions from a Hardy field.
}
\author{Vitaly Bergelson and Florian Karl Richter\\
  {\small\href{mailto:bergelson.1@osu.edu}{bergelson.1@osu.edu}}
  and
  {\small\href{mailto:richter.109@osu.edu}{richter.109@osu.edu}}}
\affil{\small Department of Mathematics\\
  The Ohio State University\\
  Columbus, Ohio
  }
\date{\small \today}
\maketitle
\begin{abstract}
Let $k\in\N$ and let $f_1,\ldots,f_k$ belong to a Hardy field. We prove that under some natural conditions on the $k$-tuple $(f_1,\ldots,f_k)$ the density of the set
\begin{equation*}
\big\{n\in \N: \gcd(n,\lfloor f_1(n)\rfloor,\ldots,\lfloor f_k(n)\rfloor)=1\big\}
\end{equation*}
exists and equals $\frac{1}{\zeta(k+1)}$, where $\zeta$ is the Riemann zeta function.
\end{abstract}
\tableofcontents

\blfootnote{The first author gratefully acknowledges the support of the NSF under grant DMS-1500575.}

\section{Introduciton}
\label{sec:intro}

\epigraph{\textit{
``It\hfill~is\hfill~a\hfill~well-known\hfill~theorem\hfill~of\hfill~{\v C}eby{\v s}ev\footnotemark{}\hfill~that\hfill~the\hfill~probability\hfill~of\hfill~the\hfill~relation\hfill~$\gcd(n, m) = 1$\hfill~is\hfill~$\frac{6}{\pi^2}$.
One\hfill~can\hfill~expect\hfill~this\hfill~still\hfill~to\hfill~remain\hfill~true\hfill~if\hfill~$m = g(n)$\hfill~is\hfill~a\hfill~function\hfill~of\hfill~$n$,\hfill~provided\hfill~that\hfill~$g(n)$\hfill~does\hfill~not\hfill~preserve\hfill~arithmetic\hfill~properties\hfill~of\hfill~$n$.''}}{P. Erd{\H o}s and G. Lorentz}

\footnotetext{
The attribution of this result to {\v C}eby{\v s}ev
(\foreignlanguage{russian}{Chebysh{\"e}v}) seems not to be justified; see however the very interesting recent preprint \cite{2016arXiv160805435A} where {\v C}eby{\v s}ev's role in the popularization of this theorem is traced and analyzed.
The result itself goes back to Dirichlet (see {\cite[pp. 51 -- 66]{Dirichlet97}}
where the equivalent statement $\sum_{n=1}^N\phi(n)\sim \frac{3}{\pi^2} n^2$ is proven)
and was rediscovered multiple times -- see for example \cite{Mertens74,Cesaro81,Cesaro83,Sylvester09, Sylvester12}.
It is worth noting that it was Ces{\` a}ro who formulated this result in probabilistic terms
\cite{Cesaro81}
and also gave a probabilistic, though not totally rigorous, proof in \cite{Cesaro83}.}

The above epigraph is a quote from the introduction to a paper by Erd{\H o}s and Lorentz
\cite{EL58}, which establishes sufficient conditions for a differentiable function
$f:[1,\infty)\to\R$ of sub-linear growth to satisfy
\begin{equation}
\label{eqn:watson}\tag{1}
d\big(\big\{n\in\N: \gcd(n,\lfloor f(n)\rfloor)=1\big\}\big) ~=~ \frac{6}{\pi^2};
\end{equation}
here $d(A)$
denotes the natural density of a set $A\subset \N$.

Perhaps the earliest result of this kind is due to Watson \cite{Watson53}, who showed
that \eqref{eqn:watson} holds for $f(n)=n\alpha$, where $\alpha$
is an irrational number (see also \cite{Estermann53,Spilker00}).
Other examples of functions for which
\eqref{eqn:watson} holds are
$f(n)=n^c$,
where $c>0$, $c\notin\N$, (see \cite{LM55} for the case
$0<c<1$ and \cite{DD02} for
the general case)
and
$f(n)=\log^r(n)$ for all $r>1$ (see \cite{LM55} for the case $r>2$
and \cite{EL58} for the general case).

The purpose of this paper is to establish \eqref{eqn:watson}
for a large class of smooth functions that naturally includes
examples such as $f(n)=n^c$ or $f(n)=\log^r(n)$; this is the
class of functions belonging to a Hardy field.

Let $G$ denote the set of all germs\footnote{We define a \define{germ
at $\infty$} to be any equivalence class of functions
under the equivalence relationship
$(f\sim g) \Leftrightarrow \big(\exists t_0>0
~\text{such that}~f(t)=g(t)~\text{for all}~t\in [t_0,\infty)\big)$.} at $\infty$ of real valued functions defined
on the half-line $[1,\infty)$.
Note that $G$ forms a ring under
pointwise addition and multiplication, which we denote by
$(G,+,\cdot)$.
Any subfield of the ring $(G,+,\cdot)$ that is closed under
differentiation is called a
\define{Hardy field}. 
By abuse of language, we say that a function $f:[1,\infty)\to\R$
belongs to some Hardy field $\Hardy$ (and write $f\in\Hardy$)
if its germ at $\infty$ belongs to $\Hardy$.
See \cite{Boshernitzan81,Boshernitzan82,Boshernitzan94}
and some references therein for
more information on Hardy fields.

A classical example of a Hardy field is the class of
{logarithmico-exponential functions}\footnote{By a
\define{logarithmico-exponential function}
we mean any function $f:(0,\infty)\to\R$ that can be
obtained from constants, $\log(t)$ and $\exp(t)$
using the standard arithmetical operations $+$, $-$, $\cdot$, $\div$ and
the operation of composition.}
introduced by Hardy in \cite{Hardy12,Hardy10};
we denote this class by $\LE$. It is worth noting that for any
Hardy field $\Hardy$ there exists a
Hardy field $\Hardy'$ such that $\Hardy'\supset \LE\cup \Hardy$.

If $\Hardy$ is a Hardy field, then one has the following basic properties:
\begin{itemize}
\item
If $f\in\Hardy$ then
$\lim_{t\to\infty}f(t)$ exists (as an element in
$\R\cup\{-\infty,\infty\}$);
\item
Any non-constant $f\in\Hardy$ is eventually either strictly
increasing or strictly decreasing; any non-linear
$f\in\Hardy$ is eventually either strictly concave or strictly convex.
\item
If $f\in\Hardy$, $g\in\LE$ and $\lim_{t\to\infty}g(t)=\infty$
then there exists a Hardy field $\Hardy'$ containing $f(g(t))$.
\item
If $f\in\Hardy$, $g\in\LE$ and $\lim_{t\to\infty}f(t)=\infty$
then there exists a Hardy field $\Hardy'$ containing $g(f(t))$.
\end{itemize}
Some well known examples of functions coming from
Hardy fields are:
$$
t^c~(\forall c\in\R),~
\log(t),~
\exp(t),~
\Gamma(t),~
\zeta(t),~
\text{Li}(t),~
\sin\left(\frac{1}{t}\right),~
etc.
$$

Before formulating our main results,
we introduce some convenient notation.
We use $\log_n(t)$ to abbreviate the $n$-th iteration of
logarithms, that is, 
$\log_2(t)=\log\log(t)$, $\log_3(t)=\log\log\log(t)$ and so on.
Also, given two functions $f,g:[1,\infty)\to\R$ we will write $f(t)\prec g(t)$
if $\frac{g(t)}{f(t)}\to\infty$ as $t\to\infty$.

Let $\Hardy$ be a Hardy field and let $f\in\Hardy$.
Consider the following two conditions:
\\

\begin{enumerate}
[label=($\text{\Alph{enumi}}$),
ref=$(\text{\Alph{enumi}})$,leftmargin=*]
\item\label{condition:A}
$
\log(t)\log_4(t)\prec f(t);
$
\item\label{condition:B}
There exists $j\in\N$ such that
$
t^{j-1} \prec f(t)\prec t^{j}.
$
\end{enumerate}
\

We have the following theorem.

\begin{Theorem}
\label{thm:mainA}
Let $\Hardy$ be a Hardy field and assume
that $f\in\Hardy$ satisfies
conditions \ref{condition:A} and \ref{condition:B}.
Then the natural density of the set
\begin{equation*}
\big\{n\in \N: \gcd(n,\lfloor f(n)\rfloor)=1\big\}
\end{equation*}
exists and equals $\frac{6}{\pi^2}$.
\end{Theorem}

Examples of sequences
$(f(n))_{n\in\N}$ to which \cref{thm:mainA} applies are
$n^c~(\text{with}~ c\notin\N)$,
$\log^2(n)$,
$n^{\sqrt{3}}\log(n)$,
$\frac{n}{\log_2(n)}$,
$\log(n!)$,
$\text{Li}(n)$,
$\log(|B_{2n}|)$
(where $B_{n}$ denotes the $n$-th Bernoulli number),
and many more.

We remark that condition \ref{condition:A}
is sharp.
Indeed, it
is shown in \cite[Section 3]{EL58} that \cref{thm:mainA}
does not hold for the function
$f(t)=\log(t)\log_4(t)$, as well as for many other
functions that grow slower than $\log(t)\log_4(t)$.

As for condition \ref{condition:B}, it can
perhaps be replaced by
the following:
\\

\begin{enumerate}
[label=($\text{\Alph{enumi}}'$),
ref=$(\text{\Alph{enumi}}')$,leftmargin=*]
\setcounter{enumi}{1}
\item\label{condition:B-prime}
There exists $j\in\N$ such that $f(t)\prec t^j$
and for all polynomials $p(t)\in\Q[t]$ we have $|f(t)-p(t)|\succ \log(t)$.
\end{enumerate}
\

Condition \ref{condition:B-prime} is inspired by a theorem of Boshernitzan (cf. \cite[Theorem 1.3]{Boshernitzan94}).
However, proving \cref{thm:mainA} under conditions
\ref{condition:A} and \ref{condition:B-prime}
would certainly necessitate introduction of new ideas.

We actually prove a multi-dimensional generalization of \cref{thm:mainA}.
Let $\Hardy$ be a Hardy field and assume $f_1,\ldots,f_k\in\Hardy$.
In addition to conditions \ref{condition:A} and \ref{condition:B}
consider the following:
\begin{enumerate}
[label=($\text{\Alph{enumi}}$),
ref=$(\text{\Alph{enumi}})$,leftmargin=*]
\setcounter{enumi}{2}
\item\label{condition:C}
$\frac{f_{i+1}}{f_{i}}\succ \log_2^4(t)$ for all $i=1,\ldots,k-1$.
\end{enumerate}
\

\begin{Theorem}
\label{thm:mainB}
Let $\Hardy$ be a Hardy field and assume $f_1,\ldots,f_k\in\Hardy$
satisfy conditions \ref{condition:A}, \ref{condition:B} and \ref{condition:C}.
Then the natural density of the set
\begin{equation*}
\big\{n\in \N: \gcd(n,\lfloor f_1(n)\rfloor,\ldots,\lfloor f_k(n)\rfloor)=1\big\}
\end{equation*}
exists and equals $\frac{1}{\zeta(k+1)}$, where $\zeta$ is the Riemann zeta function.
\end{Theorem}

We would like to remark that
our proof of \cref{thm:mainB}
works for (a larger
class of) functions which have sufficiently many
derivatives and possess some other natural regularity properties.
We decided in favor of dealing with Hardy fields since
they (a) provide an ample supply of interesting examples and
(b) allow for, so to say, cleaner proofs. 
\\

The structure of the paper is as follows.
In \cref{sec:diff-hardy} we prove some differential inequalities
for functions from a Hardy field; these inequalities will play a crucial role
in the later sections.
In \cref{sec:vdC} we briefly recall van der Corput's method for estimating
exponential sums.
In \cref{sec:estimating-exp-sums} we apply van der Corput's method to
derive useful estimates for exponential sums involving functions from
a Hardy field and in \cref{sec:discrepancy} we use
a higher dimensional version of the
Erd{\H o}s-Tur{\'a}n inequality
to convert these estimates into discrepancy estimates.
In \cref{sec:further-developing-DD}
we use the estimates derived in the previous sections to
give a proof of \cref{thm:mainB}.
Finally, in \cref{sec:f-e}, we formulate some natural
open questions.

\paragraph*{Acknowledgements:}
The authors would like to thank the anonymous referees for their helpful comments and Christian Elsholtz for the efficient handling of the submission process.
 
\section{Differential inequalities for functions from a Hardy field}
\label{sec:diff-hardy}

In this section we derive some differential inequalities for
functions belonging to a Hardy field.
Similar inequalities can be found in
\cite[Subsection 2.1]{Frantzikinakis09} and in
\cite[Subsection 2.1]{BKS15arXiv}.

Given two functions $f,g:[1,\infty)\to\R$ we write $f(t)\ll g(t)$
if there exist $C>0$ and $t_0\geq 1$ such that $f(t)\leq C g(t)$ for all $t\geq t_0$.
Also, for $\ell\in\N$ we use $f^{(\ell)}(t)$ to denote the $\ell$-th
derivative of $f(t)$.

The following lemma appears in
\cite{Frantzikinakis09}.

\begin{Proposition}[see {\cite[Corollary 2.3]{Frantzikinakis09}}]
\label{lem:Frantzikinakis-Cor2.3}
Let $\Hardy$ be a Hardy field.
Suppose $f\in\Hardy$ satisfies condition
\ref{condition:B}. Then for all $\ell\in\N$
we have,
$$
\frac{f(t)}{t^\ell
\log^2(t)} \prec |f^{(\ell)}(t)|
\ll \frac{f(t)}{t^\ell}.
$$
\end{Proposition}

Next, we derive a series of lemmas
(Lemmas \ref{lem:deriv-2} -- \ref{lem:deriv-4})
which are needed for the proof of the main result of this section, \cref{prop:hardy-diff-ineq-0}.

\begin{Lemma}
\label{lem:deriv-2}
Let $m\in\N$ and let $\Hardy$ be a Hardy field. Suppose $f,g\in\Hardy$ 
satisfy
$|f(t)|\prec |g(t)|\prec |f(t)|\log^m(t)$ and
$|\log(|f(t)|)|\succ \log_2(t)$.
Then
$$
\frac{f'(t)}{f(t)}\sim \frac{g'(t)}{g(t)}.
$$
\end{Lemma}

\begin{proof}
Our goal is to show that
$$
\frac
{
~\frac{g'(t)}{g(t)}~
}
{
\frac{f'(t)}{f(t)}
}
\xrightarrow[]{t\to\infty} 1.
$$
First we note that since $\Hardy$ is a field closed under differentiation, 
the function $\frac{{g'(t)}/{g(t)}}{{f'(t)}/{f(t)}}$
is contained in $\Hardy$. From this it follows that $\lim_{t\to\infty}\frac{{g'(t)}/{g(t)}}{{f'(t)}/{f(t)}}$ exists
(as a number in $\R\cup\{-\infty,\infty\}$). 
From L'Hospital's rule we now obtain
$$
\lim_{t\to\infty}\frac
{
~\frac{g'(t)}{g(t)}~
}
{
\frac{f'(t)}{f(t)}
}
~=~
\lim_{t\to\infty}
\frac{\log(|g(t)|)}{\log(|f(t)|)}.
$$
To finish the proof we distinguish between the cases $|f(t)|\succ 1$
and $|f(t)|\prec 1$.
If $|f(t)|\succ 1$ then, using $|f(t)|\prec |g(t)|\prec |f(t)|\log^m(t)$
and $|\log(|f(t)|)|\succ \log_2(t)$, we deduce that
$$
1\leq \lim_{t\to\infty}
\frac{\log(|g(t)|)}{\log(|f(t)|)}
\leq \lim_{t\to\infty}
\frac{\log(|f(t)|)+m\log_2(t)}{\log(|f(t)|)}=1.
$$
Likewise, if $|f(t)|\prec 1$, then we have
$$
1\geq \lim_{t\to\infty}
\frac{\log(|g(t)|)}{\log(|f(t)|)}
\geq \lim_{t\to\infty}
\frac{\log(|f(t)|)+m\log_2(t)}{\log(|f(t)|)}=1.
$$
This finishes the proof.
\end{proof}

\begin{Lemma}
\label{lem:deriv-5}
Let $\Hardy$ be a Hardy field and 
suppose $f\in\Hardy$ 
satisfies condition \ref{condition:B}. Then
$f^{(\ell)}$ satisfies either $f^{(\ell)}(t)\succ 1$
or $f^{(\ell)}(t)\prec 1$ for all $\ell\in\N$.
\end{Lemma}

\begin{proof}
By way of contradiction, let us assume that
there exist $\ell\in\N$ and $c\in\R$ such that $f^{(\ell)}(t)\sim c$.
Observe that $c\neq 0$, because otherwise $f(t)$ is a polynomial,
which contradicts condition \ref{condition:B}.

Using \cref{lem:Frantzikinakis-Cor2.3} we deduce that
$$
\frac{f(t)}{t^\ell \log^2(t)}\prec |c|\ll \frac{f(t)}{t^\ell},
$$
which is equivalent to
$$
t^\ell \ll f(t) \prec  t^\ell \log^2(t).
$$
It follows from
condition \ref{condition:B} that we can replace
$t^\ell \ll f(t)$ with $t^\ell \prec f(t)$. Therefore, we have
$$
t^\ell \prec f(t) \prec t^\ell \log^2(t).
$$

By using induction on $i$ and by repeatedly applying \cref{lem:deriv-2} to the functions
$f^{(i)}(t)$ and $t^i$,
we conclude that for all $i\in\{0,1,\ldots,\ell-1\}$,
$$
\frac{f^{(i+1)}(t)}{f^{(i)}(t)}\sim \frac{(\ell-1)t^{\ell-i-1}}{t^{\ell-i}}=\frac{\ell-i}{t}.
$$
In particular, this shows that
$$
\frac{f^{(\ell)}(t)}{f(t)}\sim \frac{\ell!}{t^\ell}.
$$
Finally, combing $t^\ell \prec f(t) $ and
$
\frac{f^{(\ell)}(t)}{f(t)}\sim \frac{\ell!}{t^\ell}
$
yields $f^{(\ell)}(t)\succ 1$, which contradicts $f^{(\ell)}(t)\sim c$.
\end{proof}

\begin{Lemma}
\label{lem:deriv-3}
Let $\Hardy$ be a Hardy field and 
suppose that $f\in\Hardy$ satisfies
either $f(t)\succ 1$
or $f(t)\prec 1$. Also, assume
$|\log(|f(t)|)|\ll \log_2(t)$.
Then
$$
\frac{|f(t)|}{t\log(t)
\log_2^2(t)} \prec |f'(t)|
\ll \frac{|f(t)|}{t\log(t)}.
$$
\end{Lemma}

\begin{proof}(cf. the proof of Lemma 2.1 in \cite{Frantzikinakis09}).~
By L'Hospital's rule we get
$$
\lim_{t\to\infty}\frac
{
~\frac{f'(t)}{f(t)}~
}
{
\frac{1}{t\log(t)}
}
~=~\lim_{t\to\infty}\frac{\log(|f(t)|)}{\log_2(t)}\ll 1.
$$
This proves that $|f'(t)|
\ll \frac{|f(t)|}{t\log(t)}$.

On the other hand, we have
$$
\lim_{t\to\infty}\frac
{
~\frac{f'(t)}{f(t)}~
}
{
\frac{1}{t\log(t)\log_2^2(t)}
}
~=~\lim_{t\to\infty} \log(|f(t)|) \log_2(t) =\pm\infty,
$$
which shows that
$
\frac{|f(t)|}{t\log(t)
\log_2^2(t)} \prec |f'(t)|
$.
\end{proof}

\begin{Lemma}
\label{lem:deriv-4}
Let $m\in\N$, let $\Hardy$ be a Hardy field and let
$f,g\in\Hardy$. Assume that
$f$ satisfies either $f(t)\succ 1$
or $f(t)\prec 1$ and $g$ satisfies either $g(t)\succ 1$
or $g(t)\prec 1$.
Also, assume $|f(t)|\prec |g(t)|\prec |f(t)|\log^m(t)$ and
$|\log(|f(t)|)|\ll \log_2(t)$.
Then
$$
\left|\frac{f'(t)}{f(t)}\right|\frac{1}{\log_2^2(t)}\prec \left|\frac{g'(t)}{g(t)}\right| \prec \left|\frac{f'(t)}{f(t)}\right|\log_2^2(t).
$$
\end{Lemma}

\begin{proof}
It follows form $|f(t)|\prec |g(t)|\prec |f(t)|\log^m(t)$ and
$|\log(|f(t)|)|\ll \log_2(t)$ that $|\log(|g(t)|)|\ll \log_2(t)$. Hence we can apply \cref{lem:deriv-3} to both $f$ and $g$ and obtain
$$
\frac{1}{t\log(t)
\log_2^2(t)} \prec \left|\frac{f'(t)}{f(t)}\right|
\ll \frac{1}{t\log(t)}
$$
as well as
$$
\frac{1}{t\log(t)
\log_2^2(t)} \prec \left|\frac{g'(t)}{g(t)}\right|
\ll \frac{1}{t\log(t)}.
$$
We deduce that
$$
\left|\frac{g'(t)}{g(t)}\right| \ll \frac{1}{t\log(t)} = \frac{ \log_2^2(t)}{t\log(t)
\log_2^2(t)} \prec \left|\frac{f'(t)}{f(t)}\right|\log_2^2(t).
$$
Similarly,
$$
\left|\frac{g'(t)}{g(t)}\right| \succ \frac{1}{t\log(t)
\log_2^2(t)} \gg \left|\frac{f'(t)}{f(t)}\right|\frac{1}{\log_2^2(t)}.
$$
\end{proof}

\begin{Proposition}
\label{prop:hardy-diff-ineq-0}
Let $m\in\N$,
let $\Hardy$ be a Hardy field,
let $f,g\in\Hardy$ 
and let $F:[1,\infty)\to(0,\infty)$
be an increasing function satisfying $1\prec F(t)\prec \log^m(t)$.
If $f$ and $g$ satisfy condition \ref{condition:B}
and $\frac{g(t)}{f(t)}\succ F(t)$,
then
$$
\left|
\frac{g^{(\ell)}(t)}{f^{(\ell)}(t)}
\right|
\succ \frac{F(t)}{\log_2^2(t)},
\qquad\forall\ell\in\N.
$$
\end{Proposition}

\begin{proof}
Let $\ell\in\N$ be arbitrary. We distinguish between
the following two cases.
The first case is $g(t)\succ f(t) \log^{m+2}(t)$
and the second case is $g(t)\ll f(t) \log^{m+2}(t)$.

We start with the proof of the first case. Using \cref{lem:Frantzikinakis-Cor2.3}
we obtain the estimate
$$
\left|\frac{g^{(\ell)}(t)}{f^{(\ell)}(t)}\right|\succ\frac{g(t)}{f(t)\log^2(t)},
$$
and, since $g(t)\succ f(t) \log^{m+2}(t)$, we get
$$
\frac{g(t)}{f(t)\log^2(t)}\succ \log^{m}(t)\succ F(t).
$$
Therefore $\left|\frac{g^{(\ell)}(t)}{f^{(\ell)}(t)}\right|\succ F(t)$,
which concludes the proof of case one.

Next, we deal with the second case. Consider the product
$$
\frac{f(t)}{g(t)}\left|\frac{g^{(\ell)}(t)}{f^{(\ell)}(t)}\right|=\prod_{i=0}^{\ell-1} \frac{\left|\frac{g^{(i+1)}(t)}{g^{(i)}(t)}\right|}{\left|\frac{f^{(i+1)}(t)}{f^{(i)}(t)}\right|}.
$$
In virtue of \cref{lem:deriv-5}, $f^{(i)}$ satisfies either
$f^{(i)}\succ 1$ or $f^{(i)}\prec 1$. The same is true for $g^{(i)}$. Also, it follows from
\cref{lem:Frantzikinakis-Cor2.3} that for at most one $i$ between $1$ and $\ell$
the function $f^{(i)}$ satisfies $|\log(|f^{(i)}(t)|)|\ll\log_2(t)$; for all other
$i$ between $1$ and $\ell$ the function $f^{(i)}$ must satisfy $|\log(|f^{(i)}(t)|)|\succ \log_2(t)$.
We can therefore apply \cref{lem:deriv-2} and \cref{lem:deriv-4} to deduce that for
at most one $i$ between $1$ and $\ell$ we have
$$
\frac{\left|\frac{g^{(i+1)}(t)}{g^{(i)}(t)}\right|}{\left|\frac{f^{(i+1)}(t)}{f^{(i)}(t)}\right|}\succ \log_2^2(t)
$$
and for all other $i$ we have
$$
\frac{\left|\frac{g^{(i+1)}(t)}{g^{(i)}(t)}\right|}{\left|\frac{f^{(i+1)}(t)}{f^{(i)}(t)}\right|}
\xrightarrow[]{t\to\infty}1.
$$
Therefore
$$
\frac{f(t)}{g(t)}\left|\frac{g^{(\ell)}(t)}{f^{(\ell)}(t)}\right|
=
\prod_{i=0}^{\ell-1} \frac{\left|\frac{g^{(i+1)}(t)}{g^{(i)}(t)}\right|}{\left|\frac{f^{(i+1)}(t)}{f^{(i)}(t)}\right|}
\succ\log_2^2(t).
$$
This, together with $\frac{g(t)}{f(t)}\succ F(t)$, implies
$$
\left|
\frac{g^{(\ell)}(t)}{f^{(\ell)}(t)}
\right|
\succ \frac{F(t)}{\log_2^2(t)}.
$$
\end{proof}

\section{Van der Corput's method for estimating exponential sums}
\label{sec:vdC}

We recall three classical theorems on estimating
exponential sums.
For proofs and more detailed discussion we refer the
reader to Section 2 in the book of
Graham and Kolesnik \cite{GK91}.
 
We start with the Kusmin-Landau
inequality for exponential sums (cf. \cite[Theorem 2.1]{GK91}).

\begin{Theorem}[Estimate based on $1^{st}$ derivative]
\label{thm:vdC-estimate-1}
Suppose $I\subset\R$ is an interval, $f\in C^{1}(I)$ and there exists
$\lambda>0$ such that
$\lambda < \left|f'(t)\right|< (1-\lambda)$ for all $t\in I$.
Then
$$
\left|\sum_{n\in I} e(f(n))\right|\ll\lambda^{-1}.
$$
\end{Theorem}

The next two theorems are due to van der Corput \cite{vanderCorput23,vanderCorput29}.

\begin{Theorem}[Estimate based on $2^{nd}$ derivative]
\label{thm:vdC-estimate-2}
Suppose $I\subset\R$ is an interval,
$f\in C^{2}(I)$ and there are
$\lambda>0$ and $\eta\geq 1$ such that
$\lambda < \left|f''(t)\right|\leq \eta\lambda$ for all $t\in I$.
Then
$$
\left|\sum_{n\in I} e(f(n))\right|\ll
|I|\eta\lambda^{\frac{1}{2}}+\lambda^{-\frac{1}{2}}.
$$
\end{Theorem}

\begin{Theorem}[Estimate based on $3^{rd}$ and higher derivatives]
\label{thm:vdC-estimate-3}
Suppose $I\subset\R$ is an interval, $\ell\geq 3$,
$f\in C^{\ell}(I)$ and there are
$\lambda>0$ and $\eta\geq 1$ such that
$\lambda < \left|f^{(\ell)}(t)\right|\leq \eta\lambda$ for all $t\in I$.
Let $Q:=2^{\ell-2}$.
Then
$$
\left|\sum_{n\in I} e(f(n))\right|\ll
|I|(\eta^2\lambda)^{\frac{1}{4Q-2}}+
|I|^{1-\frac{1}{2Q}}\eta^{\frac{1}{2Q}}+
|I|^{1-\frac{2}{Q}+\frac{1}{Q^2}}\lambda^{-\frac{1}{2Q}}.
$$
\end{Theorem}

\section{Deriving estimates for exponential sums involving functions from Hardy fields}
\label{sec:estimating-exp-sums}

\begin{Proposition}\label{prop:1st-estimate}
Let $\Hardy$ be a Hardy field and assume $f_1,\ldots,f_k:[1,\infty)\to\R$ are in $\Hardy$.
For $t\in[1,\infty)$ define $f(t):=(f_1(t),\ldots,f_k(t))$
and $E(t):=\min\{|f_1(t)|,\ldots,|f_k(t)|\}$.
Suppose we have:
\begin{enumerate}
[label=(\roman{enumi}),
ref=(\roman{enumi}),leftmargin=*]
\item\label{roman1}
for all $i\in\{1,\ldots,k\}$ the function $f_i$ satisfies
condition \ref{condition:B};
\item\label{roman2}
$\log_2(t)\prec \log(f_i(t))$ for all $i=1,\ldots,k$;
\item\label{roman3}
after reordering $f_1,\ldots,f_k$ if necessary, we have
$\frac{f_{i+1}}{f_{i}}\succ \log_2^4(t)$ for all $i=1,\ldots,k-1$;
\end{enumerate}
Then there exists a constant $C>0$
such that for all $M\in\N$, all
$r,s\in\left[1,\min\left\{M\log^2(M),\frac{E(M)}{\log^4(M)}\right\}\right]$
with $r\geq s$
and all $\tau\in [-\log_2^2(M),\log_2^2(M)]^k\cap\Z^k$ with
$\tau\neq(0,\ldots,0)$ we have
$$
\left|\sum_{n=M}^{2M}
e\left(\frac1r\langle f(s n),\tau\rangle\right)\right|
\leq  \frac{CM}{\log^2(M)}.
$$
\end{Proposition}

\begin{proof}
Let $\epsilon>0$,
$r,s\in\left[1,\min\left\{M\log^2(M),\frac{E(M)}{\log^4(M)}\right\}\right]$
with $r\geq s$
and $\tau\in [-\log_2^2(M),\log_2^2(M)]^k\cap\Z^k$ with
$\tau=(\tau_1,\ldots,\tau_k)\neq(0,\ldots,0)$
be arbitrary.
Let $b(t):=\frac1r\langle f(st),\tau\rangle$.
Our goal is to estimate
$$
\left|\sum_{n=M}^{2M}e\big( b(t)\big)\right|
$$
by using van der Corput's method of estimating exponential sums.
We therefore have to find convenient estimates for the derivatives
of $b(t)$ on the interval $[M,2M]$.

Let us pick
$i_0\in\{1,\ldots,k\}$ such that $\tau_{i_0}\neq 0$
and $\tau_i=0$ for all $i>i_0$.
Define $E_0(t):=|\tau_{i_0}f_{i_0}(t)|$.
Using condition \ref{roman3} we deduce
that $\tau_{i_0}f_{i_0}(t)$ is the dominating term in
the sum $\tau_1f_1(t)+\ldots+\tau_{i_0}f_{i_0}(t)$ and therefore
\begin{equation}
\label{eq:l-n-101-kol}
E_0(t)\ll |\langle f(t),\tau\rangle|\ll E_0(t),
\end{equation}
where the implied constants depend neither on $t$ nor
on the value of $\tau_1,\ldots,\tau_{i_0}$.
A similar argument also applies to the derivatives of 
$\langle f(t),\tau\rangle$. Indeed, it follows from
condition \ref{roman3} and \cref{prop:hardy-diff-ineq-0} (with $F(t)=\log_2^2(t)$)
that $\tau_{i_0}f_{i_0}^{(\ell)}(t)$ is the dominating term in
$\tau_1f_1^{(\ell)}(t)+\ldots+\tau_{i_0}f_{i_0}^{(\ell)}(t)$ and therefore
\begin{equation}
\label{eq:l-n-102}
|\tau_{i_0}f_{i_0}^{(\ell)}(t)|\ll |\langle f^{(\ell)}(t),\tau\rangle|\ll |\tau_{i_0}f_{i_0}^{(\ell)}(t)|.
\end{equation}

Next let
$u:=\inf\{c\in[0,\infty): f_{i_0}(t)\prec t^c\}$ 
and pick $d\in \R$ such that $s=(M\log^2(M))^{d}$
and $h\in \R$ such that
$r=(M\log^2(M))^{h}$. 
From the conditions on $r$ and $s$ we deduce that
$d\in \left[0,\min\{1,u\}\right]$ and
$h\in \left[d,\min\{1,u\}\right]$.
We now define
$$
\ell:=
\left\lceil u+du-h+\frac{1}{2}\right\rceil
$$
and set $x:=\ell-u-ud+h$.
Note that
$x\in\left[\frac{1}{2},\frac{3}{2}\right)$.

In view of
\cref{lem:Frantzikinakis-Cor2.3} we have
\begin{equation}
\label{eq:l-n-103}
\frac{f_{i_0}(t)}{t^\ell\log^2(t)}\ll
|f_{i_0}^{(\ell)}(t)|
\ll
 \frac{ f_{i_0}(t)}{t^\ell}.
\end{equation}
By combining equations \eqref{eq:l-n-101-kol},\eqref{eq:l-n-102} and \eqref{eq:l-n-103}
we obtain
$$
\frac{ E_0(t)}{t^\ell\log^2(t)}\ll |\langle f^{(\ell)}(t),\tau\rangle|\ll \frac{ E_0(t)}{t^\ell}.
$$
Hence the minimum of the function
$b^{(\ell)}(t)$
on the interval $[M,2M]$ is at least
$$
\gg\frac{ E_0(2sM)}{r (2M)^\ell\log^2(2s M)}
$$
whereas the maximum is at most
$$
\ll\frac{  E_0(s M)}{r M^\ell}.
$$

Since $ E_0(t)$ is eventually increasing,
we have $E_0(2s M)\gg E_0(s M)$. Also,
since $E_0(t)$ has polynomial growth and $s\leq E_0(M)$
we can estimate $\log(2sM)\ll \log(M)$. Therefore
$$
\frac{ E_0(2sM)}{r(2M)^\ell\log^2(2sM)}\gg
\frac{ E_0(s M)}{r M^\ell\log^2(M)}.
$$
If we choose
$$
\lambda:=\frac{ E_0(s M)}{r M^\ell\log^2(M)}
\qquad\text{and}
\qquad\eta:= \log^2(M)
$$
then it follows that
\begin{equation}
\label{eq:l-n}
\lambda \ll b^{(\ell)}(t) \ll \eta\lambda,
\qquad\forall t\in[M,2M].
\end{equation}

We now distinguish between the cases
$\ell=1$, $\ell=2$ and $\ell\geq 3$.

\paragraph{The case $\ell=1$:}
The case $\ell=1$ only occurs if
$$
f_{i_0}(t)\prec t^{\frac{1}{2}}.
$$
Therefore
$$
 b'(t) \ll \eta\lambda=\frac{E_0(sM)}{rM}\leq
\frac{|\tau_{i_0}f_{i_0}(sM)|}{rM}\leq
\frac{\log_2^2(M)f_{i_0}(sM)}{sM}\prec 1.
$$
This means we can apply \cref{thm:vdC-estimate-1}
and obtain
\begin{eqnarray*}
\label{eq:l-n-0}
\left|\sum_{n=M}^{2M} e(b(n))\right|
&\ll&\lambda^{-1}.
\end{eqnarray*}
For $\lambda^{-1}$ we have
$$
\lambda^{-1}=\frac{rM\log^2(M)}{ E_0(s M)}\leq
\frac{r M\log^2(M)}{E_0(M)}
\leq \frac{E(M) M}{E_0(M)\log^2(M)}.
$$
Finally, since $\frac{E(M)}{E_0(M)}\leq 1$ we have
$$
\lambda^{-1} \ll \frac{M}{\log^2(M)}.
$$
\paragraph{The case $\ell=2$:}
If $\ell=2$ then invoking \cref{thm:vdC-estimate-2} yields the estimate
\begin{equation}
\label{eq:l-n-2a}
\left|\sum_{n=M}^{2M} e(b(n))\right|
\ll
M\eta\lambda^{\frac{1}{2}}+\lambda^{-\frac{1}{2}}.
\end{equation}
Using
$(sM)^{u-\beta}\prec f_{i_0}(sM)\prec  (sM)^{u+\beta}$
for all $\beta>0$ we can bound
$\lambda$ from above and below,
\begin{equation}
\label{eq:l-n-2}
\frac{|\tau_{i_0}| s^{u-\beta}}{ r M^{2-u+\beta}\log^2(M)}
\ll
\lambda
\ll
\frac{|\tau_{i_0}| s^{u+\beta}}{r M^{2-u-\beta}\log^2(M)},
\end{equation}
taking into account that the implied constants
in the above equation depend on our choice of $\beta$.

Furthermore, since $x=\ell-u-ud+h$, $s= (M\log^2(M))^{d}$ and
$r= (M\log^2(M))^{h}$
we obtain from \eqref{eq:l-n-2} that
\begin{equation}
\label{eq:l-n-3}
\frac{1 }{M^{x+2\beta}\log^q(M)}
\ll
\lambda
\ll
\frac{\log^{q}(M)}{M^{x-2\beta}}.
\end{equation}
for some sufficiently large constant $q>1$.
We can use \eqref{eq:l-n-3} to further estimate \eqref{eq:l-n-2a}
and obtain
$$
M\eta\lambda^{\frac{1}{2}}+\lambda^{-\frac{1}{2}}
\ll
\frac{\log^{2+\frac{q}{2}}(M)M}{M^{\frac{x-2\beta}{2}}}+
\frac{M^{\frac{x+2\beta}{2}}}{\log^{\frac{q}{2}}(M)}.
$$
Finally, by choosing $\beta$ sufficiently small
and taking into account that $x\in\left[\frac{1}{2},\frac{3}{2}\right]$, we have
$$
\frac{\log^{2+\frac{q}{2}}(M)M}{M^{\frac{x-2\beta}{2}}}+
\frac{M^{\frac{x+2\beta}{2}}}{\log^{\frac{q}{2}}(M)}
\ll
\frac{M}{\log^2(M)}.
$$

\paragraph{The case $\ell\geq3$:}
The case $\ell\geq 3$ can be dealt with analogously to the case
$\ell=2$, only one must use \cref{thm:vdC-estimate-3} instead of
\cref{thm:vdC-estimate-2}. With $Q=2^{\ell-2}$, we have
\begin{eqnarray*}
\left|\sum_{n=M}^{2M} e(b(n))\right|
&\ll &
M(\eta^2\lambda)^{\frac{1}{4Q-2}}+
M^{1-\frac{1}{2Q}}\eta^{\frac{1}{2Q}}+
M^{1-\frac{2}{Q}+\frac{1}{Q^2}}\lambda^{-\frac{1}{2Q}},
\\
&\ll &
\frac{M}{\log^2(M)},
\end{eqnarray*}
which finishes the proof.
\end{proof}

\begin{Theorem}
\label{cor:f-s-to-s}
Let $\Hardy$ be a Hardy field and assume $f_1,\ldots,f_k:[1,\infty)\to\R$ are in $\Hardy$.
For $t\in[1,\infty)$ define $f(t):=(f_1(t),\ldots,f_k(t))$
and $E(t):=\min\{|f_1(t)|,\ldots,|f_k(t)|\}$.
Suppose we have:
\begin{enumerate}
[label=(\roman{enumi}),
ref=(\roman{enumi}),leftmargin=*]
\item\label{Roman1}
for all $i\in\{1,\ldots,k\}$ the function $f_i$ satisfies
condition \ref{condition:B};
\item\label{Roman2}
$\log_2(t)\prec \log(f_i(t))$ for all $i=1,\ldots,k$;
\item\label{Roman3}
after reordering $f_1,\ldots,f_k$ if necessary, we have
$\frac{f_{i+1}}{f_{i}}\succ \log_2^4(t)$ for all $i=1,\ldots,k-1$;
\end{enumerate}
Then there exists a constant $C>0$
such that for all $N\in\N$, all
$r,s\in\left[1,\min\left\{N,\frac{E(N)}{\log^5(N)}\right\}\right]$
with $r\geq s$
and all $\tau\in [-\log_2^2(M),\log_2^2(M)]^k\cap\Z^k$ with
$\tau\neq(0,\ldots,0)$ we have
$$
\left|\sum_{n=1}^{N}
e\left(\frac1r\langle f(s n),\tau\rangle\right)\right|
\leq  \frac{CN}{\log(N)}.
$$
\end{Theorem}

\begin{proof}
First we note that
$$
\left|\sum_{n=1}^{N}
e\left(\frac1r\langle f(s n),\tau\rangle\right)\right|
\leq \frac{N}{\log(N)}+
\left|\sum_{\frac{N}{\log(N)}\leq n\leq N}
e\left(\frac1r\langle f(s n),\tau\rangle\right)\right|,
$$
so it suffices to estimate the expression
$$
\left|\sum_{\frac{N}{\log(N)}\leq n\leq N}
e\left(\frac1r\langle f(s n),\tau\rangle\right)\right|.
$$
Dissect the interval $\left[\frac{N}{\log(N)},N\right]$
into $\log_2(N)$-many intervals of the form $[M,2M]$. If
$M\in\left[\frac{N}{\log(N)},N\right]$ then
$N<M\log^2(M)$ and $\frac{N}{\log^5(N)}<\frac{M}{\log^4(M)}$
and therefore
$$
\left[1,\min\left\{N,\frac{N}{\log^5(N)}\right\}\right]\subset
\left[1,\min\left\{M\log^2(M),\frac{M}{\log^4(M)}\right\}\right].
$$
Hence applying \cref{prop:1st-estimate} to each of
the $\log_2(N)$-many
intervals of the form $[M,2M]$ we get
\begin{eqnarray*}
\left|\sum_{\frac{N}{\log(N)}\leq n\leq N}
e\left(\frac1r\langle f(s n),\tau\rangle\right)\right|
&\ll &\log_2(N)\frac{M}{\log^2(M)}
\\
&\ll &\frac{N}{\log(N)}.
\end{eqnarray*}
This finishes the proof.
\end{proof}

\section{Discrepancy estimates}
\label{sec:discrepancy}

The following higher dimensional version of the
classical Erd{\H o}s-Tur{\'a}n inequality was discovered by
Sz{\"u}sz \cite{Szusz52} and independently by Koksma \cite{Koksma}.

\begin{Theorem}[see \cite{Szusz52},
\cite{Cochrane88}, \cite{Grabner} or {\cite[Theorem 1.21]{DT97}}]
\label{thm:erdos-turan-koksma}
Let $k\geq 1$, let $\vartheta_n\in[0,1)^k$, $n\in\N$, let $N\in\N$ and
let $a_1,\ldots,a_k,b_1,\ldots,b_k\in [0,1)$ with
$0\leq a_i< b_i<1$. Then
$$
\frac{\big|\{n\leq N:\vartheta_n\in[a_1,b_1]\times\ldots\times
[a_k,b_k]\}\big|}{N}
= \prod_{i=1}^k(b_i-a_i) + R_{N,k}
$$
and where for all $H\geq 1$,
$$
|R_{N,k}|\leq
C_k
\left(
\frac{1}{H+1}+\sum_{\tau\in [-H,H]^k,\atop \tau\neq (0,\ldots,0)}
\left(\prod_{i=1}^k \frac{1}{1+|\tau_i|}\right)\left|\frac{1}{N}\sum_{n=1}^{N}e\big(\langle \vartheta_n,\tau\rangle\big)\right|
\right).
$$
Here, $C_k$ is a constant which depends only on $k$.
\end{Theorem}

\begin{Theorem}
\label{thm:homogeneous-discrepancy-estimate-hardy-sequences}
Suppose $f_1,\ldots,f_k:[1,\infty)\to\R$
and $E:[0,\infty)\to(0,\infty)$ are as in the statement of \cref{cor:f-s-to-s}.
Then there exists a constant $C>0$
such that for all $N\in\N$ and all
$d \in\left[1,\min\left\{N,\frac{E(N)}{\log^5(N)}\right\}\right]$,
\begin{equation*}
\left|\big|
\big\{n\leq N: d\mid
\gcd(n,\lfloor f_1(n)\rfloor,\ldots,\lfloor f_k(n)\rfloor)\big\}
\big| - \frac{N}{d^{k+1}}\right|
\leq \frac{C N }{d \log_2^2(N)}.
\end{equation*}
\end{Theorem}

\begin{proof}
Define $\vartheta_{d,n}:=\left(\left\{\frac{f_1(dn)}{d}\right\},
\ldots,\left\{\frac{f_k(dn)}{d}\right\}\right)$, where
$\{x\}$ denotes the fractional part of a real number $x$.
We first observe that
\begin{equation*}
\begin{split}
\big|
\{1\leq n\leq N:
d\mid\gcd(n,\lfloor f_1(n)\rfloor,&\ldots,\lfloor f_k(n)\rfloor)\}
\big|=
\\
=&
\sum_{n=1}^N \1_{d\Z}(n)\1_{d\Z}(\lfloor f_1(n)\rfloor)\cdots
\1_{d\Z}(\lfloor f_k(n)\rfloor)
\\
=&\sum_{n\leq N/d}
\1_{d\Z}(\lfloor f_1(dn)\rfloor)\cdots
\1_{d\Z}(\lfloor f_k(dn)\rfloor)
\\
=&\left|\left\{1\leq n\leq \frac{N}{d}:
\vartheta_{d,n}\in\left[0,\frac1d\right)^k \right\}\right|.
\end{split}
\end{equation*}
From \cref{cor:f-s-to-s} we get that
$$
\left|\sum_{n=1}^{N}e\left(\frac1d \langle f(d n),\tau\rangle\right)\right|
\leq \frac{CN}{\log(N)},
$$
for all $d\in\left[1,\frac{E(N)}{\log^5(N)}\right]$ and
$\tau\in [-\log_2^2(N),\log_2^2(N)]^k\cap\Z^k$
with $\tau\neq(0,\ldots,0)$.

We now apply \cref{thm:erdos-turan-koksma} with
$H= \log_2^2(N)$ and obtain
$$
\left|
{\left|\left\{1\leq n\leq \frac{N}{d}: 
\vartheta_{d,n}\in \left[0,\frac1d\right)^k\right\}\right|}-
\frac{N}{d^{k+1}}\right|
$$
\begin{eqnarray*}
&\leq& 
\frac{C_k}{d}
\left(
\frac{N}{ \log_2^2(N)}+
\sum_{\tau\in [-\log_2^2(N),\log_2^2(N)]^k,\atop\tau\neq (0,\ldots,0)}
\left(\prod_{i=1}^k \frac{1}{1+|\tau_i|}\right)
\left|\sum_{n=1}^{N}e\big(\langle \vartheta_n,\tau\rangle\big)\right|
\right)
\\
&\leq& 
\frac{C_k}{d}
\left(
\frac{N}{ \log_2^2(N)}+\frac{CN}{\log(N)}
\sum_{\tau\in [-\log_2^2(N),\log_2^2(N)]^k,\atop\tau\neq (0,\ldots,0)}
\left(\prod_{i=1}^k \frac{1}{1+|\tau_i|}\right)
\right)
\\
&\ll &\frac{N }{d \log_2^2(N)}.
\end{eqnarray*}
\end{proof}

\section{Proving \cref{thm:mainB}}
\label{sec:further-developing-DD}
\begin{Proposition}
\label{thm:from-DD-and-VR-to-watson}
Let $k\in\N$. Let $\xi_1,\xi_2,\ldots$ be a sequence of positive integers
and let $E:[1,\infty)\to (0,\infty)$ be a function
that satisfies $E(N)\geq \max\{\xi_n:1\leq n\leq N\}$ and
$\log_2(t)\prec \log(E(t))$ and assume
that $E(N)$ has polynomial growth (i.e. there exists $j\in\N$
such that $E(t)\prec t^j$).
If
\begin{equation}
\label{eq:DD}
\left|\big|
\big\{n\leq N: d\mid \xi_n\big\}
\big| - \frac{N}{d^{k+1}}\right|
\ll \frac{N}{d \log_2^2(N)},
\quad\forall d\in\N\cap\left[1,\frac{E(N)}{\log^5(N)}\right]
\end{equation}
and
\begin{equation}
\label{eq:VR}
\big|
\big\{n\leq N: p\mid \xi_n \big\}
\big| 
\ll \frac{N}{p},
~\text{for all primes $p\in\left(\frac{E(N)}{\log^5(N)},E(N)\right]$},
\end{equation}
then the natural density of
$\{n\in\N: \xi_n=1\}$ exists and equals
$\frac{1}{\zeta(k+1)}$.
\end{Proposition}

Our proof of
\cref{thm:from-DD-and-VR-to-watson} is similar to
the proof of the main result in \cite{DD02}.

\begin{proof}
Define $G(N):=\log^4(t)$.
Let $D(N)$ be a slow growing function in $N$ and  let
$\Pi$ denote the \emph{primorial} of $D(N)$, that is,
$$
\Pi:=\prod\limits_{p~\text{prime},\atop p\leq D(N)} p.
$$
Here, by ``slow growing function'' we mean that $D(N)$
converges to $\infty$ as $N\to\infty$, but slowly enough so that
the inequality
$\Pi\leq \min\left\{\frac{E(N)}{\log^5(N)},\log_2^2(N)\right\}$ is satisfied for all $N\geq 1$.

Let $\mu(n)$ denote the classical M{\"obius} function: For $n\in\N$ define
$$
\mu(n)
=
\begin{cases}
1&\text{if }n=1;
\\
(-1)^j&\text{if $n=p_1\cdot p_2\cdot \ldots\cdot p_j$ where $p_1,\ldots,p_j$ are distinct primes};
\\
0&\text{otherwise}.
\end{cases}
$$
Using the identity
$$
\sum_{d\mid a}\mu(d)
=
\begin{cases}
1&\text{if}~a=1\\
0&\text{otherwise}
\end{cases}
$$
we obtain
\begin{eqnarray*}
\label{eq:from-DD2002-to-watson-1}
|\{n\leq N:  \xi_n=1\}|
&=&
\sum_{d}\mu(d)~
\big|
\big\{n\leq N: d\mid\xi_n\big\}
\big|.
\end{eqnarray*}
It will be convenient to decompose the right hand side
of the above equation into two sums $\Sigma_1+\Sigma_2$
where
$$
\Sigma_1:=\sum_{d\mid \Pi}\mu(d)~
\big|
\big\{n\leq N: d\mid  \xi_n\big\}
\big|
$$
and 
$$
\Sigma_2:=\sum_{d\nmid \Pi}\mu(d)~
\big|
\big\{n\leq N: d\mid  \xi_n\big\}
\big|.
$$
Our goal is  to show that
$\lim_{N\to\infty} \frac{1}{N}\Sigma_1 =\frac{1}{\zeta(k+1)}$
and that
$\lim_{N\to\infty} \frac{1}{N}\Sigma_2 =0$;
this will finish the proof.
%

First, let us show
$
\lim_{N\to\infty} \frac{1}{N}\Sigma_1 =\frac{1}{\zeta(k+1)}.
$
Recall that $\Pi\leq \frac{E(N)}{\log^5(N)}$.
Invoking condition \eqref{eq:DD} it thus follows that
$$
\left|\frac{1}{N}\Sigma_1
~-~\sum_{d\mid\Pi}\frac{\mu(d)}{d^{k+1}}\right|\ll
\frac{1}{\log_2^2(N)}
\sum_{d\mid\Pi}\frac{1}{d}.
$$
Since $\Pi\leq \log_2^2(N)$ it follows that
$\sum_{d\mid\Pi}\frac{1}{d}=
\mathcal{O}\left(\log_3(N)\right)$ and hence
$$
\frac{1}{\log_2^2(N)}
\sum_{d\mid\Pi}\frac{1}{d}
~\ll~ \frac{\log_3(N)}{\log_2^2(N)}~\xrightarrow[]{N\to\infty}0.
$$
Therefore
$$
\lim_{N\to\infty} \frac{1}{N}\Sigma_1
=\sum_{d=1}^\infty \frac{\mu(d)}{d^{k+1}} =\frac{1}{\zeta(k+1)}.
$$

For $\Sigma_2$ we obtain the estimate
\begin{eqnarray*}
\left|\Sigma_2\right|
&=&
\left|\sum_{p~\text{prime},\atop p> D(N)}
\sum_{1\leq d\leq E(N),\atop p\mid d}
\mu(d)~
\big|
\big\{n\leq N: d\mid\xi_n\big\}
\big|
\right|
\\
&=&
\left|
\sum_{n=1}^N~~
\sum_{p~\text{prime},\atop p> D(N)}~~
\sum_{d\mid  \xi_n,\atop p\mid d}\mu(d)
\right|
\\
&\leq &
\sum_{n=1}^N~~
\sum_{p~\text{prime},\atop p> D(N)}\left|
\sum_{d\mid \xi_n,\atop p\mid d}\mu(d)
\right|.
\end{eqnarray*}
It is well known (and easy to show) that
$$
\sum_{\substack{d\mid a,\\p\mid d}}\mu(d)
=
\begin{cases}
1&\text{if}~a=p^j~\text{for some $j\in\N$},\\
0&\text{otherwise}
\end{cases}
$$
and hence
\begin{eqnarray*}
\sum_{n=1}^N~~
\sum_{p~\text{prime},\atop p> D(N)}\left|
\sum_{d\mid \xi_n,\atop p\mid d}\mu(d)
\right|
&\leq &
\sum_{n=1}^N
\sum_{\substack{p~\text{prime},\\ p> D(N),\\ p\mid\xi_n}}1
\\
&\leq &
\sum_{\substack{p~\text{prime},\\ p> D(N)}}
\big|
\big\{n\leq N: p\mid  \xi_n\big\}
\big|.
\end{eqnarray*}
Putting everything together we obtain
$$
\left|\Sigma_2\right|\leq \sum_{\substack{p~\text{prime},\\ p> D(N)}}
\big|
\big\{n\leq N: p\mid  \xi_n\big\}
\big|.
$$
Again, we split the right hand side of the above equation
into two more manageable sums $\Sigma_{2,1}+\Sigma_{2,2}$, where
$$
\Sigma_{2,1}:=
\sum_{\substack{p~\text{prime},\\ D(N)< p \leq \frac{E(N)}{\log^5(N)}}}
\big|
\big\{n\leq N: p\mid \xi_n\big\}
\big|
$$
and
$$
\Sigma_{2,2}:=
\sum_{\substack{p~\text{prime},\\ \frac{E(N)}{\log^5(N)}< p \leq E(N)}}
\big|
\big\{n\leq N: p\mid \xi_n\big\}
\big|.
$$
Using condition \eqref{eq:DD} for the sum $\Sigma_{2,1}$
and using condition \eqref{eq:VR} for the sum $\Sigma_{2,2}$
we obtain the estimates
\begin{eqnarray*}
\frac{1}{N}
\Sigma_{2,1}
&\ll&
\sum_{\substack{p~\text{prime},\\ D(N)< p \leq \frac{E(N)}{\log^5(N)}}}
\frac{1}{p^2}
~+~
\frac{1}{\log_2^2(N)}
\sum_{p\leq E(N)}\frac{1}{p}
\end{eqnarray*}
and
$$
\frac{1}{N}
\Sigma_{2,2}~\leq~
\sum_{\substack{p~\text{prime},\\ \frac{E(N)}{\log^5(N)}< p \leq E(N)}}
\frac{1}{p}.
$$
Using
$$
\lim_{N\to\infty}\left|\sum_{\substack{p~\text{prime},\\ p \leq N}}\frac{1}{p}-\log_2 N\right|=M,
$$
(where $M$ is the Meissel-Mertens constant), we can estimate
$$
\frac{1}{\log_2^2(N)}
\sum_{p\leq E(N)}\frac{1}{p}
\ll\frac{\log_2(E(N))}{\log_2^2(N)}
\ll
\frac{1}{\log_2(N)},
$$
and
\begin{eqnarray*}
\lim_{N\to\infty} \frac{1}{N}\Sigma_{2,2}
&\ll&
\log(\log(E(N)))-\log\left(\log\left(\frac{E(N)}{\log^5(N)}\right)\right)
\\
&\ll&
-\log\left(1-\frac{\log(\log^5(N))}{\log(E(N))}\right)
\\
&\ll&
\frac{\log_2(N)}{\log(E(N))}.
\end{eqnarray*}
Therefore
$$
\lim_{N\to\infty}\frac{1}{N}|\Sigma_2|\leq 
\lim_{N\to\infty}\frac{1}{N}\Sigma_{2,1}+
\lim_{N\to\infty}\frac{1}{N}\Sigma_{2,2}=0. 
$$
\end{proof}

\begin{Corollary}
\label{cor:mainQ}
Let $\Hardy$ be a Hardy field and suppose $f_1,\ldots,f_k\in\Hardy$ satisfy
conditions \ref{condition:A} and \ref{condition:B}
and $\log_2(t)\prec \log(f_i(t))$ for all $i=1,\ldots,k$.
Also,
assume that $\frac{f_{i+1}}{f_{i}}\succ \log_2^4(t)$ for all $i=1,\ldots,k-1$.
Then the natural density of the set
\begin{equation*}
\big\{n\in \N: \gcd(n,\lfloor f_1(n)\rfloor,\ldots,\lfloor f_k(n)\rfloor)=1\big\}
\big|
\end{equation*}
exists and equals $\frac{1}{\zeta(k+1)}$.
\end{Corollary}

\begin{proof}
We define $$\xi_n:=\gcd(n,\lfloor f_1(n)\rfloor,\ldots,\lfloor f_k(n)\rfloor)$$
and $$E(t):=\min\{|f_1(t)|,\ldots,|f_k(t)|\}.$$
Trivially, $(\xi_n)_{n\in\N}$ satisfies condition \eqref{eq:VR}. Moreover, it follows from
\cref{thm:homogeneous-discrepancy-estimate-hardy-sequences}
that $(\xi_n)_{n\in\N}$ satisfies condition \eqref{eq:DD}.
Therefore, in view of
\cref{thm:from-DD-and-VR-to-watson}, we have
$d(\{n\in\N: \xi_n=1\})=\frac{1}{\zeta(k+1)}$.
\end{proof}


\begin{Lemma}
\label{lem:p-t-2-EL58}
Let $\Hardy$ be a Hardy field and suppose
$f\in\Hardy$ satisfies
$$
\log(t)\log_4(t)\prec f(t)\prec \frac{t}{\log_2(t)}.
$$
Define
$
S(N,d):=
\left|\big\{1\leq n\leq N:
d\mid \gcd(n,\lfloor f(n)\rfloor)\big\}\right|.
$
Then
\begin{equation}
\label{eq:EL-1}
\lim_{D\to\infty}\limsup_{N\to\infty}~
\frac{1}{N}\sum_{p~\text{prime}\atop D< p \leq f(N)}S(N,p)
~=~0.
\end{equation}
\end{Lemma}

We remark that \cref{lem:p-t-2-EL58}
can be derived from the proof of Theorem 2 in \cite{EL58}.
For the convenience of the reader we include a separate proof
here, where we follow the arguments used by
Erd{\H o}s and Lorentz in \cite{EL58}.

\begin{proof}

Let $N\in\N$ be arbitrary.
Define $I_m:=\{1\leq t \leq N: f(t)\in[m,m+1)\}$ and
note that $I_m$ is a finite subinterval of $\R$.
Let $k_m$ denote the length of $I_m$.
Note that the contribution of $I_m\cap \N$ to
the size of $S(N,p)$ is given by $I_m\cap (p\N)$. Hence
this contribution is zero if
$p\nmid m$ and it does not exceed $\frac{k_m}{p}+1$ otherwise. 
It follows that
\begin{eqnarray*}
\sum_{p~\text{prime}\atop D< p \leq f(N)}S(N,p)
&\leq &
\sum_{p~\text{prime}\atop D< p \leq f(N)}
\sum_{1\leq m \leq f(N)\atop p\mid m} \frac{k_m}{p}+1
\\
&\leq &
\sum_{p~\text{prime}\atop D< p \leq f(N)}
\sum_{1\leq m \leq \frac{f(N)}{p}} \frac{k_{pm}}{p}+1
\\
&\leq &
\sum_{p~\text{prime}\atop D< p \leq f(N)}
\sum_{1\leq m \leq \frac{f(N)}{p}} \frac{k_{pm}}{p}~+~
\sum_{p~\text{prime}\atop D< p \leq f(N)}\frac{f(N)}{p}.
\end{eqnarray*}
Define $\ell_0:=\lfloor \frac{f(N)}{p}\rfloor$.
We can write
$$
\sum_{p~\text{prime}\atop D< p \leq f(N)}
\sum_{1\leq m \leq \frac{f(N)}{p}} \frac{k_{pm}}{p}~+
\sum_{p~\text{prime}\atop D< p \leq f(N)}\frac{f(N)}{p}=\Sigma_1+\Sigma_2+\Sigma_3
$$
where
\begin{eqnarray*}
\Sigma_1
&:=&\sum_{p~\text{prime}\atop D< p \leq f(N)}\frac{f(N)}{p},
\\
\Sigma_2
&:=&
\sum_{p~\text{prime}\atop D< p \leq f(N)}
\sum_{1\leq m \leq \ell_0-1} \frac{k_{pm}}{p},
\\
\Sigma_3
&:=&
\sum_{p~\text{prime}\atop D< p \leq f(N)}
\frac{k_{p\ell_0}}{p}.
\end{eqnarray*}
We have to estimate each of the sums $\Sigma_1$, $\Sigma_2$
and $\Sigma_3$ individually.

We start with $\Sigma_1$.
Since $f(t)\prec \frac{t}{\log_2(t)}$, we conclude that
$$
\Sigma_1\leq f(N)\sum_{p \leq N}\frac{1}{p}
\ll f(N)\log_2(N)\prec N
$$
and therefore
\begin{equation}
\label{eq:EL-2}
\lim_{N\to\infty}\frac{1}{N}\Sigma_1 ~=~0.
\end{equation}

Next, we derive an estimate for $\Sigma_2$.
Since $f(t)$ is eventually strictly increasing,
the inverse function of $f^{-1}$ is well defined on some half-line
$[t_0,\infty)$. We have the identity $k_m=f^{-1}(m+1)-f^{-1}(m)$.
This means that using the inverse function theorem and the mean
value theorem we see that
there exists a number $\xi_m\in I_m$ such that
$$
k_m=\frac{1}{f'(\xi_m)}.
$$
In view of \cref{lem:Frantzikinakis-Cor2.3}, the derivative of
$f$ is eventually decreasing and therefore
$$
k_m\ll k_l
$$
for $l\geq m$.
From this we derive that
$$
\sum_{1\leq m \leq \ell_0-1} \frac{k_{pm}}{p}
\ll\frac{1}{p} \left(
\frac{k_p}{p}+\frac{k_{p+1}}{p}+\frac{k_{p+2}}{p}
+\ldots+\frac{k_{\ell_0 p-1}}{p}
\right)
\ll\frac{N}{p^2}.
$$ 
Hence
$$
\Sigma_2
\ll \sum_{p>D}\frac{N}{p^2}
$$
which proves that
\begin{equation}
\label{eq:EL-3}
\lim_{D\to\infty}
\lim_{N\to\infty}\frac{1}{N}\Sigma_2 ~=~0.
\end{equation}

Next, let $\epsilon>0$ be arbitrary
and define
$P':=\{p~\text{prime}: D<p <f(N)+1-\frac{\log_2(f(N))}{\epsilon}\}$
and
$P'':=\{p~\text{prime}: f(N)+1-\frac{\log_2(f(N))}{\epsilon}\leq p\leq f(N)\}$.
We now split the sum $\Sigma_3$ into two more sums $\Sigma_3'$
and $\Sigma_3''$, where
$$
\Sigma_3':=\sum_{p\in P'}
\frac{k_{p\ell_0}}{p}
$$
and
$$
\Sigma_3'':=\sum_{p\in P''}
\frac{k_{p\ell_0}}{p}.
$$
Arguing as before we have for all $p\in P'$,
\begin{eqnarray*}
k_{p\ell_0}
&\ll &\frac{1}{f(N)-p\ell_0+1}(k_{p\ell_0}
+k_{p\ell_0+1}+\ldots+k_{f(N)})
\\
&\ll &
\frac{N}{f(N)-p\ell_0+1}
\\
&=&
\frac{N\epsilon}{\log_2(f(N))}.
\end{eqnarray*}
This yields the following estimate for $\Sigma_3'$,
$$
\Sigma_3'\ll \sum_{p\leq f(N)}
\frac{N\epsilon}{\log_2(f(N))p}\ll N\epsilon.
$$
Let $\Pi:=\prod_{p\in P''}p$. Certainly,
$$
\Pi\leq (f(N))^{\frac{\log_2(N)}{\epsilon}}.
$$
Using the well known estimate
$$
\sum_{p\mid n}\frac{1}{p}\ll \log_3(n) 
$$
we obtain
$$
\Sigma_3''\leq k_{f(N)}\sum_{p\mid \Pi}\frac{1}{p} 
\ll
\frac{1}{f'(N)}\log_3(\Pi)\ll \frac{1}{f'(N)}\log_3(f(N)).
$$
Using L'Hospital's rule, we see that
$$
\frac{t f'(t)}{\log_3(f(t))}\gg \frac{f(t)}{\log(t)\log_3(f(t))}
\gg
\frac{f(t)}{\log(t)\log_3(\log^k(t))}
\gg\frac{f(t)}{\log(t)\log_4(t)},
$$
and therefore, using $\log(t)\log_4(t)\prec f(t)$, we get
$$
\frac{1}{f'(N)}\log_3(f(N))\prec N.
$$
Putting everything together yields
$$
\frac{1}{N}\Sigma_3=\frac{1}{N}\Sigma_3'+\frac{1}{N}\Sigma_3''
\ll \epsilon.
$$
Since $\epsilon$ was chosen arbitrarily, we deduce that
\begin{equation}
\label{eq:EL-4}
\lim_{N\to\infty}\frac{1}{N}\Sigma_3 ~=~0.
\end{equation}

Finally, combining equations \eqref{eq:EL-2},
\eqref{eq:EL-3} and \eqref{eq:EL-4} completes the proof.
\end{proof}

\begin{Lemma}
\label{lem:h-e-1}
Let $\Hardy$ be a Hardy field and suppose
$f_1,\ldots,f_k\in\Hardy$ satisfy
condition \ref{condition:B}.
Also assume that $f_1(t)\succ \log(t)$ and
$\frac{f_{i+1}}{f_{i}}\succ 1$ for all $i=1,\ldots,k-1$.
For $D\in\N$ define
$$
A_{D,N}:=
\big\{1\leq n\leq N:\gcd(n,\lfloor f_1(n)\rfloor,\ldots,
\lfloor f_k(n)\rfloor, D!)=1\big\}.
$$
Then
\begin{equation*}
\lim_{N\to\infty}\frac{|A_{D,N}|}{N} ~=~\sum_{d\mid D!}
\frac{\mu(d)}{d^{k+1}}.
\end{equation*}
\end{Lemma}

For the proof of \cref{lem:h-e-1}
we need the following Proposition, which is an immediate corollary of
\cite[Theorem 1.8]{Boshernitzan94}.

\begin{Proposition}
\label{prop:hardy-equidistribution}
Let $k\in\N$, let $\Hardy$ be a Hardy field and
let $g_1,\ldots,g_k\in \Hardy$
satisfy $g_1(t)\succ \log(t)$ and
$\frac{g_{i+1}}{g_{i}}\succ 1$ for all $i=1,\ldots,k-1$. Then the sequence
$\left(\left\{g_1(n)\right\},
\ldots,\left\{g_k(n)\right\}\right)$, $n\in\N$,
is equidistributed in $[0,1)^k$.
\end{Proposition}

\begin{proof}[Proof of \cref{lem:h-e-1}]
Define $\vartheta_{d,n}:=\left(\left\{\frac{f_1(dn)}{d}\right\},
\ldots,\left\{\frac{f_k(dn)}{d}\right\}\right)$, where
$\{x\}$ denotes the fractional part of a real number $x$.
We first observe that
\begin{eqnarray*}
|A_{D,N}|
&=&\sum_{d\mid D!}\mu(d)
\left|\big\{1\leq n\leq N:d\mid \gcd(n,\lfloor f_1(n)\rfloor,\ldots,
\lfloor f_k(n)\rfloor)\big\}\right|
\\
&=&
\sum_{d\mid D!}~
\sum_{n=1}^N \1_{d\Z}(n)\1_{d\Z}(\lfloor f_1(n)\rfloor)\cdots
\1_{d\Z}(\lfloor f_k(n)\rfloor)
\\
&=&
\sum_{d\mid D!}~
\sum_{n\leq N/d}
\1_{d\Z}(\lfloor f_1(dn)\rfloor)\cdots
\1_{d\Z}(\lfloor f_k(dn)\rfloor)
\\
&=&
\sum_{d\mid D!}~
\left|\left\{1\leq n\leq \frac{N}{d}:
\vartheta_{d,n}\in\left[0,\frac1d\right)^k \right\}\right|.
\end{eqnarray*}
Applying \cref{prop:hardy-equidistribution}
to the functions $g_1(t)=\frac{f_1(dt)}{d},
\ldots,g_k(t)=\frac{f_k(dt)}{d}$,
we deduce that 
$$
\lim_{N\to\infty}
\left|\left\{1\leq n\leq \frac{N}{d}:
\vartheta_{d,n}\in\left[0,\frac1d\right)^k \right\}\right|
~=~\frac{1}{d^{k+1}},
$$
which proves the claim.
\end{proof}

\begin{Theorem}
\label{thm:EL}
Let $\Hardy$ be a Hardy field and suppose
$f_1,\ldots,f_k\in\Hardy$ satisfy
conditions \ref{condition:A} and \ref{condition:B}. Also
assume that $f_1(t)\prec \frac{t}{\log_2(t)}$ and 
$\frac{f_{i+1}}{f_{i}}\succ 1$ for all $i=1,\ldots,k-1$.
Then the natural density of the set
\begin{equation*}
\big\{n\in \N: \gcd(n,\lfloor f_1(n)\rfloor,\ldots,\lfloor f_k(n)\rfloor)=1\big\}
\big|
\end{equation*}
exists and equals $\frac{1}{\zeta(k+1)}$.
\end{Theorem}

\begin{proof}
Let $D\in\N$.
We define
$$
A_N:=
\big\{1\leq n\leq N:\gcd(n,\lfloor f_1(n)\rfloor,\ldots,
\lfloor f_k(n)\rfloor)=1\big\}
$$
and
$$
A_{D,N}:=
\big\{1\leq n\leq N:\gcd(n,\lfloor f_1(n)\rfloor,\ldots,
\lfloor f_k(n)\rfloor, D!)=1\big\}.
$$
Certainly, $A_N\subset A_{D,N}$. It follows directly from
\cref{lem:h-e-1} that
$$
\lim_{D\to\infty}\lim_{N\to\infty}\frac{|A_{D,N}|}{N}=
\frac{1}{\zeta(k+1)}.
$$
Hence, it suffices to show that
$$
\lim_{D\to\infty}\limsup_{N\to\infty}\frac{|A_{D,N}\setminus A_N|}{N}
=0.
$$

Let
$$
S'(N,d):=\left|
\big\{1\leq n\leq N:
d\mid \gcd(n,\lfloor f_1(n)\rfloor,\ldots,\lfloor
f_k(n)\rfloor)\big\}\right|
$$
and let
$$
S(N,d):=\left|
\big\{1\leq n\leq N:
d\mid \gcd(n,\lfloor f_1(n)\rfloor)\big\}\right|.
$$
It is clear that
$$
|A_{D,N}\setminus A_N|\leq 
\sum_{p~\text{prime}\atop D< p \leq \lfloor f_1(n)\rfloor}
S'(N,p).
$$
However, we have that
$$
\sum_{p~\text{prime}\atop D< p \leq \lfloor f_1(n)\rfloor}
S'(N,p)
\leq
\sum_{p~\text{prime}\atop D< p \leq \lfloor f_1(n)\rfloor}
S(N,p)
$$
and it follows from
\cref{lem:p-t-2-EL58} that
$$
\lim_{D\to\infty}\limsup_{N\to\infty}
\frac{1}{N}~\sum_{p~\text{prime}\atop D< p \leq \lfloor f_1(n)\rfloor}
S(N,p)~=~0.
$$
\end{proof}

\begin{proof}[Proof of \cref{thm:mainB}]
Let $f_1,\ldots,f_k\in\Hardy$
satisfy conditions \ref{condition:A}, \ref{condition:B} and
\ref{condition:C}.

If $f_1(t)\prec \frac{t}{\log_2(t)}$ then the conclusion of \cref{thm:mainB}
follows from \cref{thm:EL}. On the other hand, if
$f_1(t)\gg \frac{t}{\log_2(t)}$ then the conclusion of \cref{thm:mainB}
follows from \cref{cor:mainQ}.
\end{proof}

\section{Some open questions}
\label{sec:f-e}

We end this paper with formulating some open questions.

\paragraph{\arabic{section}.1.}
The first question concerns
a natural extension of Watson's \cite{Watson53} original result.

\begin{Question}
Let $\alpha_1,\ldots,\alpha_k$ be $k$ irrational numbers.
Is it true that the natural density of the set
\begin{equation*}
\left\{n\in \N: \gcd\left(n,\lfloor n\alpha_1 \rfloor, \lfloor n^2\alpha_2 \rfloor,
\ldots,\lfloor n^k\alpha_k \rfloor\right)=1\right\}
\end{equation*}
exists and equals $\frac{1}{\zeta(k+1)}$?
\end{Question}

\paragraph{\arabic{section}.2.}
Let $\Hardy$ be a Hardy field, let $f_1,\ldots,f_k\in\Hardy$
and consider the condition:
\begin{enumerate}
[label=($\text{\Alph{enumi}}'$),
ref=($\text{\normalfont\Alph{enumi}}'$),leftmargin=*]
\setcounter{enumi}{2}
\item\label{condition:C-prime}
$\frac{f_{i+1}}{f_{i}}\succ 1$ for all $i=1,\ldots,k-1$.
\end{enumerate}

\begin{Question}
In the statement of \cref{thm:mainB},
can one replace condition \ref{condition:C} with
condition \ref{condition:C-prime}? 
\end{Question}

\paragraph{\arabic{section}.3.}

By slightly generalizing the methods
used by Estermann in \cite{Estermann53},
one can prove
the following theorem:
\begin{Theorem}
For any $k$-tuple $(\alpha_1,\ldots,\alpha_k)$ of rationally independent irrational numbers
the natural density of the set
\begin{equation*}
\big\{n\in \N: \gcd(\lfloor n\alpha_1 \rfloor,\ldots,\lfloor n\alpha_k \rfloor)=1\big\}
\end{equation*}
exists and equals $\frac{1}{\zeta(k)}$.
\end{Theorem}

This leads to the following question.

\begin{Question}
Let $\Hardy$ be a Hardy field.
Suppose $f_1,\ldots,f_k \in\Hardy$ satisfy
conditions \ref{condition:A}, \ref{condition:B} and \ref{condition:C-prime}.
Is it true that the natural density of the set
\begin{equation*}
\big\{n\in \N: \gcd(\lfloor f_1(n) \rfloor,\ldots,\lfloor f_k(n) \rfloor)=1\big\}
\end{equation*}
exists and equals $\frac{1}{\zeta(k)}$?
\end{Question}

\paragraph{\arabic{section}.4.}

We conclude this paper with a question which addresses a possible weakening of condition \ref{condition:B} in Theorems \ref{thm:mainA} and \ref{thm:mainB}.

\begin{Question}
Can condition \ref{condition:B} be replaced by condition \ref{condition:B-prime} (see page \pageref{condition:B-prime}) in Theorems \ref{thm:mainA} and \ref{thm:mainB}?
\end{Question}


\bibliographystyle{siam}

\providecommand{\noopsort}[1]{} 

\allowdisplaybreaks
\small
\bibliography{DatabaseW}

\begin{thebibliography}{10}

\bibitem{2016arXiv160805435A}
{\sc S.~{Abramovich} and Y.~Y. {Nikitin}}, {\em {On the probability of
  co-primality of two natural numbers chosen at random (Who was the first to
  pose and solve this problem?)}}, ArXiv e-prints,  (2016).
\newblock \url{http://arxiv.org/abs/1608.05435}.

\bibitem{BKS15arXiv}
{\sc V.~Bergelson, G.~Kolesnik, and Y.~Son}, {\em Uniform distribution of
  subpolynomial functions along primes and applications}, ArXiv e-prints,
  (2015).
\newblock \url{http://arxiv.org/abs/1503.04960}; to appear in J. Anal. Math.

\bibitem{Boshernitzan81}
{\sc M.~Boshernitzan}, {\em An extension of {H}ardy's class {$L$} of ``orders
  of infinity''}, J. Analyse Math., 39 (1981), pp.~235--255.

\bibitem{Boshernitzan82}
\leavevmode\vrule height 2pt depth -1.6pt width 23pt, {\em New ``orders of
  infinity''}, J. Analyse Math., 41 (1982), pp.~130--167.

\bibitem{Boshernitzan94}
{\sc M.~D. Boshernitzan}, {\em Uniform distribution and {H}ardy fields}, J.
  Anal. Math., 62 (1994), pp.~225--240.

\bibitem{Cesaro81}
{\sc E.~Ces{\`a}ro}, {\em Questions propos{\'e}es, \# 75}, Mathesis, 1 (1981),
  p.~184.

\bibitem{Cesaro83}
\leavevmode\vrule height 2pt depth -1.6pt width 23pt, {\em Solutions de
  questions propos{\'e}es, question 75}, Mathesis, 3 (1983), pp.~224--225.

\bibitem{Cochrane88}
{\sc T.~Cochrane}, {\em Trigonometric approximation and uniform distribution
  modulo one}, Proc. Amer. Math. Soc., 103 (1988), pp.~695--702.

\bibitem{DD02}
{\sc F.~Delmer and J.-M. Deshouillers}, {\em On the probability that {$n$} and
  {$[n^c]$} are coprime}, Period. Math. Hungar., 45 (2002), pp.~15--20.

\bibitem{Dirichlet97}
{\sc G.~Dirichlet}, {\em Mathematische {W}erke. Band II}, Herausgegeben auf
  Veranlassung der K\"oniglich Preussischen Akademie der Wissenschaften von L.
  Kronecker, Druck und Verlag von Georg Reimer, 1897.

\bibitem{DT97}
{\sc M.~Drmota and R.~F. Tichy}, {\em Sequences, discrepancies and
  applications}, vol.~1651 of Lecture Notes in Mathematics, Springer-Verlag,
  Berlin, 1997.

\bibitem{EL58}
{\sc P.~Erd{\H{o}}s and G.~G. Lorentz}, {\em On the probability that {$n$} and
  {$g(n)$} are relatively prime}, Acta Arith., 5 (1959), pp.~35--44.

\bibitem{Estermann53}
{\sc T.~Estermann}, {\em On the number of primitive lattice points in a
  parallelogram}, Canadian J. Math., 5 (1953), pp.~456--459.

\bibitem{Frantzikinakis09}
{\sc N.~Frantzikinakis}, {\em Equidistribution of sparse sequences on
  nilmanifolds}, J. Anal. Math., 109 (2009), pp.~353--395.

\bibitem{Grabner}
{\sc P.~J. Grabner}, {\em Erd{\H o}s-{T}ur\'an type discrepancy bounds},
  Monatsh. Math., 111 (1991), pp.~127--135.

\bibitem{GK91}
{\sc S.~W. Graham and G.~Kolesnik}, {\em {V}an der {C}orput's method of
  exponential sums}, vol.~126 of London Mathematical Society Lecture Note
  Series, Cambridge University Press, Cambridge, 1991.

\bibitem{Hardy12}
{\sc G.~H. Hardy}, {\em Properties of {L}ogarithmico-{E}xponential
  {F}unctions}, Proc. London Math. Soc., S2-10 (1912), pp.~54--90.

\bibitem{Hardy10}
\leavevmode\vrule height 2pt depth -1.6pt width 23pt, {\em Orders of infinity.
  The {\it Infinit\"arcalc\"ul}\ of Paul du Bois-Reymond}, Hafner Publishing
  Co., New York, 1971.
\newblock Reprint of the 1910 edition, Cambridge Tracts in Mathematics and
  Mathematical Physics, No. 12.

\bibitem{Koksma}
{\sc J.~F. Koksma}, {\em Some theorems on {D}iophantine inequalities}, Scriptum
  no. 5, Math. Centrum Amsterdam, 1950.

\bibitem{LM55}
{\sc J.~Lambek and L.~Moser}, {\em On integers {$n$} relatively prime to
  {$f(n)$}}, Canad. J. Math., 7 (1955), pp.~155--158.

\bibitem{Mertens74}
{\sc F.~Mertens}, {\em {\"U}ber einige asymptotische {G}esetze der
  {Z}ahlentheorie}, J. Reine Angew. Math., 77 (1874), pp.~289--338.

\bibitem{Spilker00}
{\sc J.~Spilker}, {\em Die {F}astperiodizit\"at der {W}atson-{F}unktion}, Arch.
  Math. (Basel), 74 (2000), pp.~26--29.

\bibitem{Sylvester09}
{\sc J.~J. Sylvester}, {\em The Collected mathematical papers. Volume III
  (1870--1883)}, Cambridge University Press, 1909.
\newblock pp. 672--676.

\bibitem{Sylvester12}
\leavevmode\vrule height 2pt depth -1.6pt width 23pt, {\em The Collected
  mathematical papers. Volume IV (1882--1897)}, Cambridge University Press,
  1912.
\newblock pp. 84--87.

\bibitem{Szusz52}
{\sc P.~Sz{\"u}sz}, {\em \"{U}ber ein {P}roblem der {G}leichverteilung}, in
  Comptes {R}endus du {P}remier {C}ongr\`es des {M}ath\'ematiciens {H}ongrois,
  27 {A}o\^ut--2 {S}eptembre 1950, Akad\'emiai Kiad\'o, Budapest, 1952,
  pp.~461--472.

\bibitem{vanderCorput23}
{\sc J.~G. van~der Corput}, {\em Neue zahlentheoretische {A}bsch\"atzungen},
  Math. Ann., 89 (1923), pp.~215--254.

\bibitem{vanderCorput29}
\leavevmode\vrule height 2pt depth -1.6pt width 23pt, {\em {N}eue
  zahlentheoretische {A}bsch\"atzungen ({Z}weite {M}itteilung)}, Math. Z., 29
  (1929), pp.~397--426.

\bibitem{Watson53}
{\sc G.~L. Watson}, {\em On integers {$n$} relatively prime to {$[\alpha n]$}},
  Canadian J. Math., 5 (1953), pp.~451--455.

\end{thebibliography}

\bigskip
\footnotesize

\noindent
Vitaly Bergelson\\
\textsc{Department of Mathematics, Ohio State University, Columbus OH-43210, USA}\par\nopagebreak
\noindent
\href{mailto:bergelson.1@osu.edu}
{\texttt{bergelson.1@osu.edu}}
\\

\noindent
Florian K.\ Richter\\
\textsc{Department of Mathematics, Ohio State University, Columbus OH-43210, USA}\par\nopagebreak
\noindent
\href{mailto:richter.109@osu.edu}
{\texttt{richter.109@osu.edu}}

\end{document}